\documentclass[preprint,12pt]{elsarticle}

\usepackage{amssymb}
\usepackage{amsthm}
\usepackage{graphics}
\usepackage{epsfig}
\usepackage{amssymb}
\usepackage{graphicx}
\usepackage{subfigure}
\usepackage{color}
\usepackage{tikz}
\usetikzlibrary{patterns}
\usepackage{hyperref}
\usepackage{a4wide}
\usepackage{amsmath}
\usepackage{amsfonts}
\usepackage{enumerate}

\newtheorem*{theoremaux}{Theorem \theoremauxnum}
\gdef\theoremauxnum{1}

\newtheorem{lemma}{\bf Lemma}[section]
 \newtheorem{example}{\bf Example}[section]
\newtheorem{theorem}{\bf Theorem}[section]
\newtheorem{proposition}[lemma]{\bf Proposition}
\newtheorem{corollary}[lemma]{\bf Corollary}
\newtheorem{definition}{\bf Definition}[section]
\newtheorem{remark}{\bf Remark}[section]

\journal{~}

\begin{document}

\begin{frontmatter}



\title{On Co-Maximal Subgroup Graph of a Group}




\author{Angsuman Das\corref{cor1}}
\ead{angsumandas054@gmail.com}

\author{Manideepa Saha}
\ead{manideepasaha1991@gmail.com }
\address{Department of Mathematics, Presidency University, Kolkata, India} 

\author{Saba Al-Kaseasbeh}
\ead{saba.alkaseasbeh@gmail.com}
\address{Department of Mathematics, Faculty of Science, Tafila Technical University, Tafila, Jordan}

\cortext[cor1]{Corresponding author}

\begin{abstract}
The co-maximal subgroup graph $\Gamma(G)$ of a group $G$ is a graph whose vertices are non-trivial proper subgroups of $G$ and two vertices $H$ and $K$ are adjacent if $HK=G$. In this paper, we continue the study of $\Gamma(G)$, especially when $\Gamma(G)$ has isolated vertices. We define a new graph $\Gamma^*(G)$, which is obtained by removing isolated vertices from $\Gamma(G)$. We characterize when $\Gamma^*(G)$ is connected, a complete graph, star graph, has an universal vertex etc. We also find various graph parameters like diameter, girth, bipartiteness etc. in terms of properties of $G$.
\end{abstract}


\end{frontmatter}

\section{Introduction}
The idea of associating graphs with groups, which started from Cayley graphs, is now an important topic of research in modern algebraic graph theory. The most prominent graphs defined on groups in recent years are power graphs \cite{Cameron-Ghosh}, enhanced power graphs \cite{enhanced-power-graph}, commuting graphs \cite{commuting-graph}, non-commuting graphs \cite{non-commuting-graph}, subgroup inclusion graphs \cite{subgroup-inclusion-graph} etc., and various works like \cite{hiranya} has been done on these topics. As a comprehensive survey on different graphs defined on groups, \cite{Cameron-survey} is an excellent reference. These graphs help us in understanding various group properties using graph theoretic interpretations. Following these footsteps, Akbari {\it et.al.} introduced {\it co-maximal subgroup graph} of a group $G$ in \cite{akbari}. The comaximal subgroup graph of $\mathbb{Z}_n$ has been studied in \cite{manideepa-sucharita-angsu}.

\begin{definition} Let $G$ be a group and $S$ be the collection of all non-trivial proper subgroups of $G$. The co-maximal subgroup graph $\Gamma(G)$ of a group $G$ is defined to be a graph with $S$ as the set of vertices and two distinct vertices $H$ and $K$ are adjacent if and only if $HK=G$.	
\end{definition}
Although the definition of subgroup product graph allows the possibility of $G$ being infinite, in this paper, we restrict ourselves to finite groups only. However if the results translate similarly to infinite groups, we will mention it separately. Note that the definition implies that the graph is undirected as $HK=G$ if and only if $KH=G$. In this paper, we continue the study of co-maximal subgroup graph of a group. We also introduced deleted co-maximal subgroup graph $\Gamma^*(G)$ which is obtained by removing the isolated vertices of $\Gamma(G)$. We study the existence of isolated vertices of $\Gamma(G)$, connectedness of $\Gamma^*(G)$ and characterize various properties of $\Gamma(G)$ and $\Gamma^*(G)$.

\subsection{Preliminaries}
We first recall some definitions and results on graph theory and group theory. For undefined terms and results on group theory, please refer to \cite{prime-power-groups-book} and \cite{rotman-book} and that of graph theory, please refer to \cite{west-graph-book}.
 
 Let $\Gamma$ be a graph. The {\it diameter} of a connected graph $\Gamma$ is the maximum distance between any two vertices in $\Gamma$. The minimum degree of a vertex in $\Gamma$ is denoted by $\delta(\Gamma)$ and $girth(\Gamma)$ denotes the length of a smallest cycle in $\Gamma$.

Let $G$ be a finite group. Denote by $\pi(G)$, the set of prime divisors of $|G|$. A proper subgroup $H$ of a group $G$ is said to be a {\it maximal subgroup} if there does not exist any proper subgroup of $G$ which properly contains $H$. A group $G$ is said to be {\it minimal non-cyclic} if $G$ is non-cyclic but every proper subgroup of $G$ is cyclic. 
The set of all maximal subgroups of the group $G$ is denoted by $Max(G)$. The {\it Frattini subgroup} of a group $G$ is defined as the intersection of all maximal subgroups of $G$ and is denoted by $\Phi(G)$. The {\it intersection number} of a finite group $G$, denoted by $\iota(G)$, is the minimum number of maximal subgroups of $G$ whose intersection is equal to $\Phi(G)$. By $D_{2n}$, we mean the dihedral group of order $2n$.

Now, we state a few standard group theoretic results which we will be using throughout the paper.

\begin{theorem}[Miller and Moreno, 1903 (\cite{miller-moreno},\cite{minimal-non-cyclic-paper},Theorem 1)]\label{minimal-non-cyclic}
	A group $G$ is a minimal non-cyclic group if and only if $G$ is isomorphic to one of the following groups:
	\begin{enumerate}
		\item $\mathbb{Z}_p\times \mathbb{Z}_p$.
		\item The quaternion group $Q_8$ of order $8$
		\item $\langle a,b| a^p=b^{q^m}=1,b^{-1}ab=a^r \rangle$, where $p$ and $q$ are distinct primes and $r\not\equiv 1~(mod~p), r^q\equiv 1~(mod~p)$. 		
	\end{enumerate}	
\end{theorem}

\begin{proposition}\label{preli-result}
	Let $G$ be a finite group.
	\begin{enumerate}
		\item If $G$ has a unique maximal subgroup, then $G$ is cyclic $p$-group.
		\item If $G$ has exactly two maximal subgroups, then $G$ is cyclic and $|G|=p^aq^b$, where $p,q$ are distinct primes.
		\item $G$ is nilpotent if and only if all maximal subgroups of $G$ are normal in $G$.
		\item $|G/\Phi(G)|$ is divisible by all primes $p$ dividing $|G|$.
	\end{enumerate}
\end{proposition}

\subsection{Previous Works}
In \cite{akbari}, authors proved various results on $\Gamma(G)$. We recall a few of them, which will be used in this paper. 

\begin{theorem}[\cite{akbari}, Theorem 2.2]
	Let $G$ be a group. If $\delta(\Gamma(G))\geq 1$, then $diam(\Gamma(G)) \leq 3$.
\end{theorem}

\begin{theorem}[\cite{akbari}, Theorem 2.4]
Let $G$ be a finite group with at least two proper non-trivial subgroups. Then the following are equivalent:\begin{enumerate}
	\item $\Gamma(G)$ is connected.
	\item $\delta(\Gamma(G))\geq 1$.
	\item $G$ is supersolvable and its Sylow subgroups are all elementary abelian.
	\item $G$ is isomorphic to a subgroup of a direct product of groups of squarefree order.
\end{enumerate}
\end{theorem}

\begin{corollary}\cite{akbari}
	Let $G$ be a nilpotent group. Then $\Gamma(G)$ is connected if and only if
	$\Phi(G) = \{e\}$ or $G \cong \mathbb{Z}_{p^2}$, for some prime number $p$.
\end{corollary}

\begin{theorem}\label{universal-akbari}[\cite{akbari}, Theorem 3.5]
	Let $G$ be a nilpotent group. Then the following are equivalent.
	\begin{enumerate}
		\item There exists a vertex adjacent to all other vertices of $\Gamma(G)$.
		\item $G \cong \mathbb{Z}_p\times \mathbb{Z}_q$, where $p$ and $q$ are (not necessarily distinct) primes.
		\item $\Gamma(G)$ is complete.
	\end{enumerate}
\end{theorem}

\section{Isolated Vertices in $\Gamma(G)$}
In \cite{akbari}, authors mainly focussed on graphs $\Gamma(G)$ with $\delta(\Gamma(G))\geq 1$, i.e., graphs without isolated vertices. The only discussion on isolated vertices appear in Remark 2.9 \cite{akbari}, where they characterized isolated vertices in the case when $G$ is abelian.

We start with some examples of $\Gamma(G)$, both connected and disconnected.
\begin{example} 
	Consider the Klein-$4$ group, $K_4$. Then the set of all proper nontrivial subgroups is  $S=\{H_1=\{e,a\}, H_2=\{e,b\}, H_3=\{e, ab\}\}$ and the corresponding $\Gamma(K_4)$ is given in Figure \ref{examplefigure}(A). Next, consider  the group $S_3$. Then the set of all proper nontrivial subgroups is  $S=\{ H_1=\{e,(1 2)\}, H_2= \{e,(1 3)\} , H_3=\{e,(2 3)\}, H_4= \{e,(1 2 3), (1 3 2)\} \}$ and the corresponding $\Gamma(S_3)$ is given in Figure \ref{examplefigure}(B). 
	\begin{figure}[!ht]
		\begin{center}
			\includegraphics[scale=0.35]{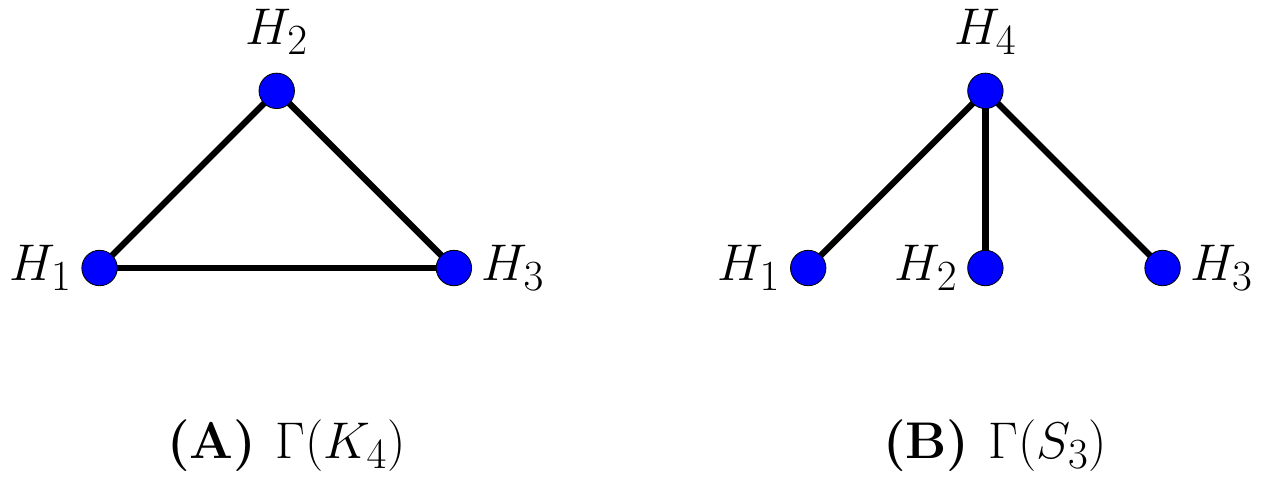}
		\end{center}
		\caption{Examples of $\Gamma(G)$ (Connected Examples)}
		\label{examplefigure}
	\end{figure}	
\end{example}

\begin{example} 
	Consider the Quaternion group, $Q_8=\langle a,b: a^4=e, a^2=b^2,ba=a^3b \rangle$. Then the set of all proper nontrivial subgroups is  $S=\{H_1=\langle a^2 \rangle, H_2=\langle a \rangle, H_3=\langle ab \rangle, H_4=\langle b \rangle \}$ and the corresponding $\Gamma(Q_8)$ is given in Figure \ref{examplefigure2}(A). Next, consider the Dihedral group $D_8=\langle a,b: a^4=e, b^2=e,ba=a^3b \rangle$. Then the set of all proper nontrivial subgroups is  $S=\{ H_1=\langle a^2 \rangle, H_2=\langle b \rangle , H_3=\langle ab \rangle, H_4=\langle a^2b \rangle , H_5=\langle a^3b \rangle, T_1=\langle a \rangle, T_2=\{e,a^2,b,a^2b\},T_3=\{e,ab,a^2,a^3b\} \}$ and the corresponding $\Gamma(D_8)$ is given in Figure \ref{examplefigure2}(B). 
	\begin{figure}[!ht]
		\begin{center}
			\includegraphics[scale=0.35]{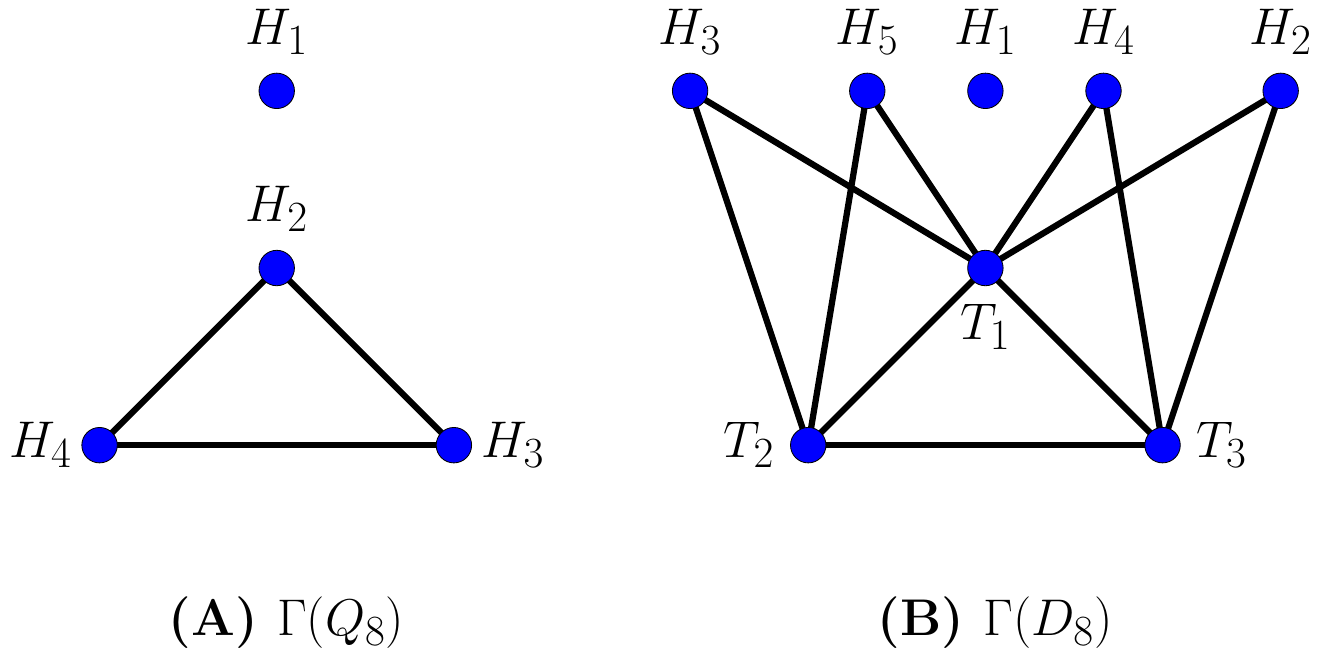}
		\end{center}
		\caption{Examples of $\Gamma(G)$ (Disconnected Examples)}
		\label{examplefigure2}
	\end{figure}	
\end{example}	

\begin{theorem}\label{isolated-vertex-theorem}
	Let $G$ be a finite group. If $\{e\}\subsetneq H\subseteq \Phi(G)$, then $H$ is an isolated vertex in $\Gamma(G)$. Conversely, if $G$ is nilpotent and $H$ is an isolated vertex in $\Gamma(G)$, then $H\subseteq \Phi(G)$.
\end{theorem}
\begin{proof}
Let $H$ be a non-trivial proper subgroup of $G$	which is contained in $\Phi(G)$. If possible, let $H\sim K$ in $\Gamma(G)$ for some non-trivial proper subgroup $K$ of $G$. Then there exists a maximal subgroup $M$ of $G$ which contains $K$. Thus, $G=HK\subseteq HM=M\neq G$, a contradiction. Thus $H$ is an isolated vertex in $\Gamma(G)$.

Conversely, let $H$ be an isolated vertex in $\Gamma(G)$. If possible, let $M$ be a maximal subgroup of $G$ which does not contain $H$. As $G$ is nilpotent, by Proposition \ref{preli-result}(3), $M$ is a normal subgroup of $G$ and hence $HM$ is a subgroup of $G$ and $M\subsetneq HM$. Thus, by maximality of $M$, we have $HM=G$, i.e., $H\sim M$ in $\Gamma(G)$, a contradiction. Thus $H$ is contained in every maximal subgroup of $G$, i.e., $H\subseteq \Phi(G)$.
\end{proof}

{\remark \label{isolated-remark} The above theorem proves that if the Frattini subgroup of $G$ is non-trivial, then $\Gamma(G)$ is disconnected. It is to be noted that if $G$ is not nilpotent, triviality of the Frattini subgroup of $G$ does not imply $\Gamma(G)$ is connected. For example, $A_4$ is not nilpotent and $\Phi(A_4)$ is trivial. But $\Gamma(A_4)$ is the disjoint union of a star $K_{1,4}$ and three isolated vertices. This shows that solvability of $G$ is not enough for the converse to hold. We can even say something more: super-solvability is also not enough. Consider the Frobenius group $G=\langle a,b: a^5=b^4=1, ab=ba^2 \rangle$ of order $20$. It is super-solvable, non-nilpotent group with $|\Phi(G)|=1$, but it has five isolated vertices. 

}

\begin{theorem}\label{edgeless}
	Let $G$ be a finite group. If $G$ is a cyclic $p$-group, then $\Gamma(G)$ has no edges. Conversely, if $G$ is a solvable group such that $\Gamma(G)$ has no edges, then $G$ is a cyclic $p$-group.
\end{theorem}

\begin{proof}
 If $G$ is a cyclic $p$-group, then $G\cong \mathbb{Z}_{p^k}$ and $S=\{\langle p\rangle, \langle p^2\rangle, \cdots, \langle p^{k-1}\rangle \}$ is the set of vertices of $\Gamma(G)$. Clearly, all of the vertices are isolated in $\Gamma(G)$. 
 
For the other  direction, since $G$ is a solvable group, by a Theorem of Hall (See \cite{rotman-book}, Theorem 5.28, pp. 108), $G$ is the Zappa-Szep product of a Sylow $p$-subgroup $H$ and a Hall $p'$-subgroup $K$, i.e., $G=HK$ and $H\cap K=\{e\}$. If $K$ is non-trivial, then we have an edge $(H,K)$ in $\Gamma(G)$. Thus $K$ must be trivial. Hence $G$ is a finite $p$-group of order, say $p^n$. By Sylow's theorem, $G$ has a subgroup $N$ of order $p^{n-1}$ and $N$ is normal in $G$. Let $g \in G \setminus N$ and $A=\langle g \rangle$. If $A$ is a proper subgroup of $G$, then $AN$ is a subgroup of $G$ containing $N$ and $p^{n-1}=|N|<|AN|$, i.e., $|AN|=p^n$, i.e., $AN=G$. Thus we get an edge $(N,A)$ in $\Gamma(G)$, a contradiction. Hence $A=G=\langle g \rangle$, i.e., $G$ is a cyclic $p$-group. 
\end{proof}	

{\remark The solvability of $G$ is required for the converse to hold: If $G=PSL(2,13)$, then $\Gamma(G)$ consists of only isolated vertices.}

In Theorem 2.2 of \cite{akbari}, the authors proved that if $\Gamma(G)$ has no isolated vertices, then it is connected and its diameter is bounded above by $3$. In the next theorem, we discuss about the components of $\Gamma(G)$, if $\Gamma(G)$ has isolated vertices.

\begin{theorem}\label{connected-component}
	Let $G$ be a finite group such that $G$ has a maximal subgroup which is normal in $G$. Then $\Gamma(G)$ is connected apart from some possible isolated vertices. Moreover, the diameter of the component is less than or equal to $4$.
\end{theorem}
\begin{proof}
	If $G$ is a cyclic $p$-group, then $\Gamma(G)$ is edgeless and hence the result holds. So, we assume that $G$ is not a cyclic $p$-group. Let $H$ be a maximal subgroup of $G$ which is normal in $G$. 
	As $G$ is not a cyclic $p$-group, there exists another maximal subgroup $L$ of $G$. Thus $HL=G$ and $H \sim L$, i.e., $H$ is not an isolated vertex in $\Gamma(G)$. Let $C_1$ be the component of $\Gamma(G)$ which contains $H$. If possible, let there exists another component $C_2$ of $\Gamma(G)$ and $H'\sim K'$ in $C_2$.
	
	If both $H',K' \subseteq H$, then $H'K'\subseteq H \neq G$, which contradicts that $H'\sim K'$. Thus, at least one of $H'$ and $K'$ is not contained in $H$. Let $H'$ is not contained in $H$. Then $H \subsetneq HH'=G$ and hence $H \sim H'$ in $\Gamma(G)$. This contradicts that $H$ and $H'$ are in different component of $\Gamma(G)$. Hence, our assumption is wrong and $\Gamma(G)$ is connected apart from some possible isolated vertices.
	
	Now, we prove the upper bound for the diameter. If $H$ is the only maximal subgroup of $G$, then $G$ is a cyclic $p$-group, and the resulting graph $\Gamma(G)$ is empty. So, we assume that there exist other maximal subgroups of $G$, apart from $H$. Let $A,B$ be two arbitrary vertices of the component. If $A,B\not\subseteq H$, then we have $A\sim H\sim B$, i.e., $d(A,B)\leq 2$.
	
	If $A\not\subseteq H$ and $B \subseteq H$, then as $B$ is not an isolated vertex in the component, there exists a subgroup $B'$ of $G$ such that $B'\sim B$ in $\Gamma(G)$. Clearly $B'$ is contained in some maximal subgroup $M$ of $G$, with $M\neq H$. Thus $B\sim M \sim H \sim A$, i.e., $d(A,B)\leq 3$.
	
	Lastly, let us assume $A,B \subseteq H$. Clearly $A \not\sim B$ in $\Gamma(G)$. As $A$ and $B$ are not isolated vertices in the component, there exist subgroups $A'$ and $B'$ such that $A \sim A'$ and $B \sim B'$ in $\Gamma(G)$. Again, $A'$ and $B'$ are contained in some maximal subgroups $M_A$ and $M_B$ respectively, where $H\neq M_A,M_B$. If $M_A=M_B$, then $A \sim M_A\sim B$, i.e., $d(A,B)=2$. If $M_A\neq M_B$, then we have $A \sim M_A \sim H \sim M_B \sim B$, i.e., $d(A,B)\leq 4$.
\end{proof}	

\begin{corollary}\label{nilpotent-diameter-bound}
Let $G$ be a finite nilpotent group. Then $\Gamma(G)$ is connected apart from some possible isolated vertices and the component has diameter at most $3$.
\end{corollary}
\begin{proof}
As every maximal subgroup in a finite nilpotent group $G$ is normal in $G$, by Theorem \ref{connected-component}, $\Gamma(G)$ is connected apart from some possible isolated vertices. Now, we prove that the diameter of the unique connected component of $\Gamma(G)$ is less than or equal to $3$. If $G$ has a unique maximal subgroup, then, by Lemma \ref{preli-result}(1), $G$ is a cyclic $p$-group and hence $\Gamma(G)$ is edgeless. So, we assume that $G$ has at least two distinct maximal subgroups and we denote the set of all maximal subgroups of $G$ by $\mathcal{M}$. Let $H$ and $K$ be two vertices in the component of $\Gamma(G)$. Then,  $H$ and $K$ are not isolated vertices and by Theorem \ref{isolated-vertex-theorem}, $H,K \not\subseteq \Phi(G),$	i.e., there exists maximal subgroups $M_1$ and $M_2$ in $\mathcal{M}$ such that $H \not\subseteq M_1$ and $K \not\subseteq M_2$. If $M_1=M_2$, then $M_1\subsetneq HM_1=G=KM_1\supsetneq M_1$ (since $G$ is nilpotent and $M_1$ is a maximal subgroup), i.e., we have a path $H\sim M_1\sim K$ in $\Gamma(G)$ and $d(H,K)\leq 2$. Similarly, if $M_1\neq M_2$, then we have a path $H \sim M_1 \sim M_2 \sim K$, i.e., $d(H,K)\leq 3$. Thus the diameter of the component of $\Gamma(G)$ is less than or equal to $3$.
\end{proof}

\begin{corollary}
Let $G$ be a finite solvable group. Then $\Gamma(G)$ is connected apart from some possible isolated vertices and the component has diameter at most $4$.
\end{corollary}
\begin{proof}
It suffices to prove that a finite solvable group $G$ always have a maximal subgroup which is normal in $G$. For finite groups, an equivalent definition of solvability is as follows: A finite group $G$ is solvable if there are subgroups $\{e\}=G_0 \trianglelefteq G_1 \trianglelefteq \cdots \trianglelefteq G_{k-1}\trianglelefteq G_k=G$ such that each factors $G_{i+1}/G_i$ is a cyclic group of prime order. So in particular, $G/G_{k-1}$ is a cyclic group of prime order. Hence $G_{k-1}$ is a maximal subgroup of $G$ which is normal in $G$.
\end{proof}

{\remark\label{diam=4-remark} In Theorem \ref{connected-component}, we have proved that the diameter of the connected component is at most $4$. However, we are yet to find an example of a group $G$, where the diameter of the component is equal to $4$.}

Theorem \ref{isolated-vertex-theorem} and Theorem \ref{connected-component}, motivates us to put forward the next definition. 
\begin{definition}
	Let $G$ be a group. The {\it deleted co-maximal subgroup graph} of $G$, denoted by $\Gamma^*(G)$, is defined as the graph obtained by removing the isolated vertices from $\Gamma(G)$.
\end{definition}
Thus we have the immediate corollary:
\begin{corollary}\label{Gamma*-connected}
	Let $G$ be a finite solvable group. Then $\Gamma^*(G)$ is connected and $diam(\Gamma^*(G))\leq 4$. If $G$ is nilpotent, then $diam(\Gamma^*(G))\leq 3$ and $\Gamma^*(G)=\Gamma(G)$ if and only if $\Phi(G)$ is trivial.
\end{corollary}

\begin{remark} \label{non-solvable-connected-remark}
	There exists groups like $S_n$ with $n\geq 5$ which are not solvable but has a maximal subgroup $A_n$ which is normal in $S_n$. Thus there exists finite non-solvable groups $G$ such that $\Gamma^*(G)$ is connected. Presently, authors are not aware of any finite non-solvable group $G$ such that $\Gamma^*(G)$ is disconnected. 
\end{remark}

\section{Some Characterizations of $\Gamma(G)$ and $\Gamma^*(G)$}
In this section, we characterize some graph properties like completeness, bipartiteness, girth etc. of $\Gamma(G)$ and $\Gamma^*(G)$. We note that authors in \cite{akbari}, (See Theorem 3.5 in \cite{akbari} or Theorem \ref{universal-akbari}) proved that if $G$ is nilpotent, then $\Gamma(G)$ has an universal vertex if and only if $G \cong \mathbb{Z}_p\times \mathbb{Z}_q$. However, in this section, we prove that they are not equivalent in general, i.e., if $G$ is not nilpotent, there are other groups whose co-maximal subgroup graph has universal vertex. In fact, we characterize when $\Gamma(G)$ and $\Gamma^*(G)$ has an universal vertex, is complete, is a star graph etc.

\begin{theorem}\label{complete} $\Gamma(G)$ is a complete graph on more than one vertices if and only if $G\cong \mathbb{Z}_p \times \mathbb{Z}_q$, where $p$ and $q$ are (not necessarily distinct) primes.
\end{theorem}
\begin{proof} We first assume that $G\cong \mathbb{Z}_p \times \mathbb{Z}_q$. If $p\neq q$, then $\Gamma(G)\cong \mathbb{Z}_2$ and hence complete. If $p=q$, then $G$ has exactly $p+1$ subgroups of order $p$, say $H_1,H_2,\cdots, H_{p+1}$ and no other proper nontrivial  subgroups. Hence $S=\{H_1,H_2,\cdots, H_{p+1}\}$. Note that $|H_iH_j|=\frac{|H_i||H_j|}{|H_i\cap H_j|}=p^2$. Thus $H_iH_j=G$ for $i\neq j$. Thus $\Gamma(G)$ is a complete graph of order $p+1$. 
	
	Conversely, we assume that  $\Gamma(G)=K_n$ of order $n>2$ and let $H_1,H_2,\cdots, H_n $ be the subgroups of $G$ which form a complete graph. Clearly, $H_i \nsubseteq H_j$  and $H_i \cap H_j=\{e\}$
	for all $i,j$ with $ i \neq j$, because if $H_i \subseteq H_j$, then $H_iH_j  \subseteq H_j \neq G$, i.e $H_i $ is not adjacent to $H_j$ which contradicts by our assumption. Also, if $ H_i \cap H_j=K \neq \{e\}$, then $H_i K \subseteq H_i\neq G$, i.e  $H_i $ is not adjacent to $K$ which contradicts by our assumption. Similarly $H_i$'s  do not have any proper subgroup and hence $H_i$'s  are prime order subgroups of $G$ of order $p_i$. Now as $H_i $ is adjacent to $H_j$ for all $i,j$ with $ i \neq j$, we have $H_iH_j =G$, i.e $|G|=|H_i||H_j |=p_ip_j$ for all $i,j$ with $ i \neq j$. Thus all $p_i$'s equal, say $p_i=p$ for all $i$. Therefore, $|G|=p^2$. So $G=\mathbb{Z}_{p^2}$ or $G=\mathbb{Z}_p \times \mathbb{Z}_p$. Since $\Gamma(G)$ has no isolated vertex, $G$ is not cyclic. Hence $G\cong \mathbb{Z}_p \times \mathbb{Z}_p$ and $n=p+1$, so $p=n -1$ is a prime. If $\Gamma(G)=K_2$, then it trivially follows that $G\cong \mathbb{Z}_p \times \mathbb{Z}_q$, for distinct primes $p$ and $q$.
\end{proof}

\begin{theorem}\label{star-graph} Let $G$ be a group of order $n$. $\Gamma(G)$ is a star graph $ K_{1,p}$ if and only if $G$ is a group of order $n=pq$, where $p>q$ are distinct primes.
\end{theorem}

\begin{proof} Let $\Gamma(G)$ be a star $K_{1,p}$, where $H$ is the universal vertex and $ K_1,K_2, \cdots,K_p$ are leaves [see Figure \ref{star-proof}]. If $H$ has any proper nontrivial subgroup $L$, then $HL\neq G$. Therefore $ H$ is not adjacent to $L$. However as $H$ is a universal vertex, no such $L$  exist. Thus $H$ has no non-trivial proper subgroup, it means $H$  is a subgroup of prime order, say $\eta$. Also, we have $HK_i=G$, for all $i=1,2,\cdots, p$. Thus $n=|G|=\frac{|H||K_i|}{|H\cap K_i|}= \eta |K_i|$. Hence, $|K_i|=\frac{n}{\eta}$,  for all $i=1,2,\cdots, p$.
	
	Note that as all $K_i$'s are of the same order, they are not contained in each other. Moreover, $K_i$'s  can not have any non-trivial proper subgroups (otherwise if $L \subsetneq K_i$ is a subgroup, then $L$ must be adjacent to $H$, i.e., $L=K_j$. But that means $|K_i|>|L|=|K_j|=|K_i|$, a contradiction). Therefore, $K_i$'s are prime order subgroups. Hence $\frac{n}{\eta}$ is prime, say $q$. So $n=q\eta$. Note that $q \neq \eta$, as otherwise $n=q^2$ and $G=\mathbb{Z}_{q^2}$ or $G=\mathbb{Z}_q \times \mathbb{Z}_q$. In none of the case, $\Gamma(G)$ is a star. Thus, $G$ has $ p$  subgroups $ K_1, K_2, \cdots, K_p$ each has order $q$. By Sylow's Theorem, the number of Sylow $q$-subgroups, $n_q$ is given by $p=n_q=1+tq|\eta$. Hence, $p|\eta$. So $p=\eta$ and $n=pq$.
	Also, as $ K_1$ is not adjacent to $K_2$. we have $K_1K_2 \neq G$. So, $pq=|G|>|K_1K_2|=\frac{|K_1||K_2|}{|K_1\cap K_2|}=q^2$. It means $p>q$. 
	
	For the other direction, let $G$ be a group of order $pq$, with $p>q$. Then $G \cong \mathbb{Z}_{pq}$ or $G\cong \mathbb{Z}_p \rtimes \mathbb{Z}_q$  and $q|(p-1)$. In the first case, we get $\Gamma(G)\cong P_2$, which is a star. In the second case, it is easy to see that $G$ has a normal subgroup $H$ of order $p$ and $p$ subgroups $ K_1, K_2, \cdots, K_p$ each has order $q$. Now, it is easy to check that $\Gamma(G)$ is a star with $H$ as the universal vertex.
	\begin{figure}[!ht]
		\begin{center}
			\includegraphics[scale=0.35]{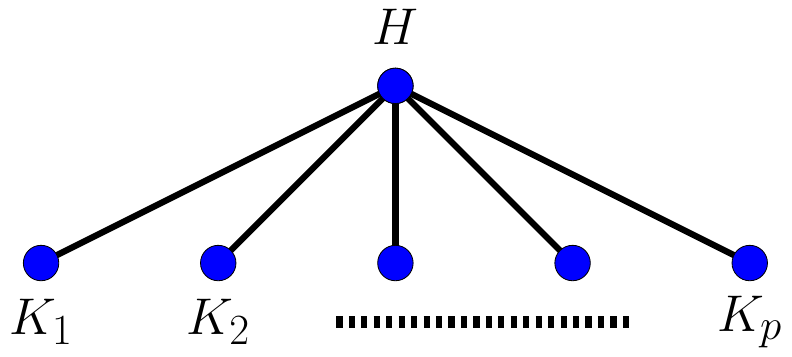}
		\end{center}
		\caption{$\Gamma(G)$ is a star graph $ K_{1,p}$}
		\label{star-proof}
	\end{figure} 		
\end{proof}

\begin{remark}
From \cite{akbari} Theorem 6.2 and Theorem \ref{star-graph}, it follows that for a finite group $G$, $\Gamma(G)$ is a star if and only if $\Gamma(G)$ is a tree if and only if $G$ is group of order $pq$, where $p$ and $q$ are distinct primes.
\end{remark}

\begin{theorem}\label{Gamma(G)-universal}
	Let $G$ be a finite group. Then $\Gamma(G)$ has an universal vertex if and only if either $G$ is non-cyclic abelian group of order $p^2$ or $G$ is group of order $pq$, where $p$ and $q$ are distinct primes.
\end{theorem}
\begin{proof}
	Let $H$ be a non-trivial proper subgroup of $G$ such that $H$ is an universal vertex in $\Gamma(G)$. Clearly $H$ is both maximal and minimal subgroup in $G$, as otherwise $H$ fails to be an universal vertex. Thus $H$ is a prime order subgroup of $G$ of order, say $p$. Thus $|G|=pm$. Clearly $m\neq 1$, i.e., $|G|\neq p$, as in that case $\Gamma(G)$ has no vertex.
	
	If $p|m$, then $H$ is contained in some Sylow $p$-subgroup $K$ of $G$ and hence $HK=K$. If $K \neq G$, then we get a contradiction as $H$ is an universal vertex in $\Gamma(G)$. If $K=G$, then $G$ is a $p$-group. Let $|G|=p^k$. If $k\geq 4$, then $G$ has a subgroup $L$ of order $p^2$ and it is easy to see that $|HL|\leq p^3$, i.e., $H\not\sim L$ in $\Gamma(G)$, a contradiction. Thus $|G|=p^2$ or $p^3$. If $G$ is cyclic, then, by Theorem \ref{edgeless}(1), $\Gamma(G)$ is edgeless, a contradiction. Thus $G$ is either a non-cyclic abelian group of order $p^2$ or a non-cyclic group of $p^3$.
	
	If $G$ is a non-cyclic abelian group of order $p^2$, then $G \cong \mathbb{Z}_p \times \mathbb{Z}_p$ and by Theorem \ref{complete}, $\Gamma(G)$ is a complete graph. If $G$ is a non-cyclic group of order $p^3$, then there are exactly two non-abelian groups and two abelian groups of order $p^3$, upto isomorphism. We discuss each of these possibilities separately:
	
	{\bf Abelian Groups:} In this case, $G \cong \mathbb{Z}_p \times \mathbb{Z}_p \times \mathbb{Z}_p$ or $G\cong \mathbb{Z}_p\times \mathbb{Z}_{p^2}$. In both cases, for every subgroup of order $p$, we can find another subgroup of order $p$ whose product is not equal to $G$, i.e., $\Gamma(G)$ does not have any universal vertex.
	
	{\bf Non-Abelian Groups:} In this case, it is well known that the Frattini subgroup of a non-abelian group of order $p^3$ is non-trivial. Thus it always have isolated vertices, a contradiction.
	
	
	Next we consider the case when $p\nmid m$. If $m$ has two distinct prime factors $q$ and $r$, then, by Cauchy's theorem, there exist a subgroup $K$ of order $q$ and $|HK|=pq<pm$. Even if $m=q^t$ where $q$ is a prime and $t\geq 2$, we get a subgroup $K$ of order $q$ and $|HK|=pq<pm$. So only possibility is $m=q$, i.e., $|G|=pq$. Let $p>q$. If $q\nmid (p-1)$, then $G$ is cyclic and $\Gamma(G)\cong K_2$. If $q\mid (p-1)$, then $G$ is either cyclic or a non-abelian group of order $pq$. In the latter case, $G$ has unique normal subgroup $H$ of order $p$ and all other vertices of $\Gamma(G)$ are adjacent to $H$.
	
	Combining all the cases, it follows that if $\Gamma(G)$ has a universal vertex, then either $G$ is non-cyclic abelian group of order $p^2$ or $G$ is group of order $pq$, where $p$ and $q$ are distinct primes.
	
	The converse follows from Theorem \ref{complete} and Theorem \ref{star-graph}.	
\end{proof}	

\begin{theorem}
	Let $G$ be a nilpotent group. $\Gamma^*(G)$ is a star if and only if either $G$ is a group of order $pq$ or the cyclic group of order $p^rq$, where $p,q$ are distinct primes and $r\geq 2$ is an integer.
\end{theorem}
\begin{proof}
	If $G$ is a cyclic group of order $p^rq$, then $G\cong \mathbb{Z}_{p^rq}$ and $\Gamma(G)$ is the union of isolated vertices $\langle pq \rangle,\langle p^2q \rangle,\ldots,\langle p^{r-1}q \rangle$ and a star, where $\langle q \rangle$ is the universal vertex and $\langle p \rangle,\langle p^2 \rangle,\ldots,\langle p^{r-1} \rangle$ are the leaves. If $G$ is a group of order $pq$, then by Theorem \ref{star-graph}, $\Gamma(G)=\Gamma^*(G)$ is a star.
	
	Conversely, let $G$ be a nilpotent group such that $\Gamma^*(G)$ is a star. Then two cases may occur.
	
	{\bf Case 1. $\Gamma(G)=\Gamma^*(G)$:} Then by Theorem \ref{star-graph}, $G$ is a group of order $n=pq$.
	
	{\bf Case 2. $\Gamma(G)\neq \Gamma^*(G)$:} This means $\Gamma(G)$ has at least one isolated vertex. Let $H$ be the universal vertex of the star $\Gamma^*(G)$ and $K_1,K_2,\ldots,K_t$ be the leaves of the star. 
	
	{\bf Claim 1:} $H$ is a maximal subgroup of $G$: If it is not, then there exists a subgroup $L$ of $G$ such that $H\subsetneq L\subsetneq G$. However, this impies $G=HK_i\subsetneq LK_i$, i.e., $LK_i=G$, i.e., $H \neq L$ and $L\sim K_i$ for $i=1,2,\ldots,t$, which contradicts that it is a star.	
	
	{\bf Claim 2:} $G$ is a cyclic group: If $G$ has a unique maximal subgroup $H$, then by Lemma  \ref{preli-result}(1), $G$ is a cyclic $p$-group, which implies that $\Gamma(G)$ is edgeless, a contradiction. Thus $G$ has at least one maximal subgroup $M$, other than $H$. As $G$ is nilpotent, both $H$ and $M$, being maximal subgroups, are normal in $G$ and $H\subsetneq HM=G$, i.e., $H \sim M$ in $\Gamma(G)$. However, this means $M=K_i$ for some $i \in \{1,2,\ldots,t\}$. Now, choose $a \in G\setminus (H\cup M)$ (note that $G\neq H \cup M$) and set $A=\langle a \rangle$. If $A$ is a proper subgroup of $G$, then $M\subsetneq MA=G$, i.e., $K_i=M\sim A$. Now, as $K_i$ is a leaf, $A$ must be the universal vertex $H$, i.e., $A=H$. However, as $a \notin H$, we have $A\neq H$, a contradiction. Hence, $A$ is not a proper subgroup of $G$, i.e., $A=G=\langle a\rangle$ is cyclic.
	
	Let $|G|={p_1}^{\alpha_1}{p_2}^{\alpha_2}\cdots {p_k}^{\alpha_k}$. If $k\geq 3$, then we get a clique of size $k$, namely $\langle p_1\rangle,\langle p_2\rangle , \ldots, \langle p_k\rangle$, a contradiction. Also $k=1$ implies that $\Gamma(G)$ is edgeless. Thus $k=2$, i.e., $|G|={p_1}^{\alpha_1}{p_2}^{\alpha_2}$. If possible, let both $\alpha_1,\alpha_2 \geq 2$. Then, we get a $4$-cycle, namely $\langle p_1\rangle,\langle p_2\rangle,\langle p^2_1\rangle,\langle p^2_2\rangle$, a contradiction. If both $\alpha_1=\alpha_2=1$, then $\Gamma(G)\cong K_2$, i.e., without any isolated point. Thus, $|G|$ is of the required form.
\end{proof}	

\begin{theorem}\label{Gamma*-complete}
	Let $G$ be a nilpotent group. $\Gamma^*(G)$ is a complete graph if and only if either $G$ is isomorphic to $\mathbb{Z}_p\times \mathbb{Z}_p$ or $Q_8$.
\end{theorem}
\begin{proof}
	If $G\cong \mathbb{Z}_p\times \mathbb{Z}_p$ or $Q_8$, then the result follows from Theorem \ref{complete} and Figure \ref{examplefigure2}(A).	
	
	Conversely, let $\Gamma^*(G)$ be a complete graph. If $\Gamma(G)=\Gamma^*(G)$, then by Theorem \ref{complete}, $G \cong \mathbb{Z}_p\times \mathbb{Z}_p$ for some prime $p$. So, we assume that $\Gamma(G)\neq \Gamma^*(G)$, i.e., $\Gamma(G)$ has at least one isolated vertex. Then by Theorem \ref{isolated-vertex-theorem}(2), the Frattini subgroup of $G$, $\Phi(G)$ is non-trivial. Let $H_1,H_2,\ldots,H_n$ be the vertices of $\Gamma^*(G)$.
	
	{\bf Claim 1:} Each $H_i$ is a maximal subgroup of $G$.
	
	{\it Proof of Claim 1:}	Let there exists a proper subgroup $L$ with $H_i\subsetneq L$. As $H_i\sim H_j$ in $\Gamma^*(G)$, we have $LH_j\supsetneq H_iH_j=G$, i.e. $L\sim H_j$, i.e., $L$ is a vertex of $\Gamma^*(G)$. But $LH_i\subsetneq L\neq G$, which contradicts that $\Gamma^*(G)$ is complete. Hence $H_i$ is a maximal subgroup of $G$.
	
	Thus, all the maximal subgroups of $G$ are vertices of $\Gamma^*(G)$ and there exists at least two maximal subgroups of $G$. If the number of maximal subgroups is exactly $2$, then by Proposition \ref{preli-result}(2), $G \cong \mathbb{Z}_{p^aq^b}$. Now as $\Gamma(G)$ has at least one isolated vertex, we have $ab>1$, i.e., either $a>1$ or $b>1$. Suppose $a>1$. Then $\langle p^2 \rangle$ is not an isolated vertex and $\langle p \rangle \not\sim \langle p^2 \rangle$ in $\Gamma(G)$. This contradicts that $\Gamma^*(G)$ is complete. Thus the number of maximal subgroups of $G$ is at least $3$.
	
	{\bf Claim 2:} Each $H_i$ is cyclic.
	
	{\it Proof of Claim 2:} Let $a\in H_i\setminus \Phi(G)$ and $A=\langle a \rangle$. If $A\subsetneq H_i$, then we get a proper subgroup of $G$ which is not contained in $\Phi(G)$. Thus $A$ is a vertex of $\Gamma^*(G)$. But $AH_i\subseteq H_i\neq G$, i.e., $A\not\sim H_i$, which contradicts that $\Gamma^*(G)$ is complete. Thus $A=H_i=\langle a\rangle$, i.e., $H_i$ is cyclic.  
	
	As all the maximal subgroups of $G$ are cyclic, it follows that all subgroups of $G$ are cyclic. Next we prove that all non-maximal subgroups are contained in $\Phi(G)$. 
	
	{\bf Claim 3:} Any proper subgroup of $H_i$ is contained in $\Phi(G)$.
	
	{\it Proof of Claim 3:} If possible, let there exists a subgroup $L$ of $G$ such that $\Phi(G)\subsetneq L \subsetneq H_i$. Then $L$ is a vertex of $\Gamma^*(G)$ and $LH_i\subseteq H_i\neq G$, i.e., $L\not\sim H_i$, a contradiction, as $\Gamma^*(G)$ is complete.
	
	As $\Phi(G)$ is normal in $G$, it is also normal in each $H_i$. Now, it follows from the above claim that $H_i/\Phi(G)$ has no non-trivial proper subgroup, i.e., $|H_i/\Phi(G)|=p_i$, for some prime $p_i$, for $i=1,2,\ldots,n$. 
	
	Now, as $H_i\sim H_j$ in $\Gamma^*(G)$, we have $H_iH_j=G$, i.e., $$|G|=|H_iH_j|=\dfrac{|H_i||H_j|}{|H_i\cap H_j|}=\dfrac{p_i|\Phi(G)|\cdot p_j|\Phi(G)|}{|\Phi(G)|}=p_ip_j|\Phi(G)|.$$
	Since $G$ has at least $3$ maximal subgroups, we have $p_i=p_j=p$(say) for all $i,j \in \{1,2,\ldots,n\}$. Hence $|G|=p^2|\Phi(G)|$ and $|H_i|=p|\Phi(G)|$, i.e., $|G/\Phi(G)|=p^2$. Hence, by Proposition \ref{preli-result}(4), $G$ is a $p$-group. Let $|G|=p^k$. Then $|H_i|=p^{k-1}$.

	Now, we try to classify the group $G$. If $G$ is abelian, then as $G$ has cyclic subgroups of order $p^{k-1}$, we have $G\cong \mathbb{Z}_{p^k}$ or $\mathbb{Z}_{p^{k-1}}\times \mathbb{Z}_p$. The former can not hold as in that case $\Gamma(G)$ is edgeless. Thus $G\cong \mathbb{Z}_{p^{k-1}}\times \mathbb{Z}_p$. However, $\langle p \rangle \times \mathbb{Z}_p$, being a subgroup of order $p^{k-1}$ of $\mathbb{Z}_{p^{k-1}}\times \mathbb{Z}_p$, is a maximal subgroup which is not cyclic, a contradiction. Thus $G\not\cong \mathbb{Z}_{p^{k-1}}\times \mathbb{Z}_p$. 
	
	Hence $G$ is a non-abelian group of order $p^k$. Then $G$ is a minimal non-cyclic $p$-group. Thus, by Theorem \ref{minimal-non-cyclic}, $G$ is isomorphic to $Q_8$. (As $G$ is a non-abelian $p$-group, other two cases do not occur.)
\end{proof}

\begin{theorem}\label{Gamma*-universal}
	Let $G$ be a finite nilpotent group. $\Gamma^*(G)$ has an universal vertex if $G$ is isomorphic to one of the following groups: 
	\begin{enumerate}
		\item $\mathbb{Z}_{p^rq}$, where $p,q$ are distinct primes.
		\item $\mathbb{Z}_{p}\rtimes \mathbb{Z}_{q}$, where $p,q$ are distinct primes and $q|p-1$.
		\item $\mathbb{Z}_{p^{n-1}}\times \mathbb{Z}_{p}$, where $p$ is a prime.
		\item $M_{p^n}=\langle a,b: a^{p^{n-1}}=b^p=e; b^{-1}ab=a^{1+p^{n-2}} \rangle$, where $p$ is a prime.
		\item $D_{2^{n}}=\langle a,b: a^{2^{n-1}}=b^2=e; bab=a^{-1} \rangle$.
		\item $Q_{2^{n}}=\langle a,b: a^{2^{n-1}}=e; b^2=a^{2^{n-2}}; b^{-1}ab=a^{-1} \rangle$.
		\item $SD_{2^{n}}=\langle a,b: a^{2^{n-1}}=b^2=e; bab=a^{-1+2^{n-2}}\rangle$.
	\end{enumerate}
\end{theorem}
\begin{proof}
	Let $\Gamma^*(G)$ has an universal vertex. If $\Gamma^*(G)=\Gamma(G)$, then $G$ is as described in Theorem \ref{Gamma(G)-universal}. If $\Gamma^*(G)\neq \Gamma(G)$, then $\Phi(G)$ is non-trivial. Let $H$ be an universal vertex of $\Gamma^*(G)$.
	
	{\bf Claim 1:} $H$ is a maximal subgroup.\\
	{\it Proof of Claim 1:} If not, then there exists a proper subgroup $L$ with $H \subsetneq L\neq G$. Then $HL=L\neq G$, i.e., contradicting that $H$ is an universal vertex, unless $L$ is an isolated vertex of $\Gamma(G)$. However, in that case, $H \subsetneq L\subseteq \Phi(G)$, a contradiction.
	
	{\bf Claim 2:} $H$ is cyclic.\\
	{\it Proof of Claim 2:} If $G$ has exactly one maximal subgroup $H$, then $G$ is a cyclic $p$-group and hence $\Gamma^*(G)$ is an empty graph. Thus $G$ has at least two maximal subgroups and hence $\Phi(G)\subsetneq H$. Let $a \in H \setminus \Phi(G)$. Then $\langle a \rangle \subseteq H$ and $\langle a \rangle \not\subseteq \Phi(G)$. If $\langle a \rangle$ is a proper subgroup of $H$, then $\langle a \rangle H=H\neq G$, i.e., $\langle a \rangle \not\sim H$ in $\Gamma^*(G)$, a contradiction. Thus $H=\langle a \rangle$.
	
	Note that all subgroups of $H$ are contained in $\Phi(G)$, i.e., $H/\Phi(G)$ has no non-trivial proper subgroups. This implies that $H/\Phi(G)$ is a cyclic group of prime order, say $p$. Then $|H|=p|\Phi(G)|$. 
	
	Also note that all elements in $H \setminus \Phi(G)$ are generators of the cyclic group $H$. Thus by equating the number of generators of a finite cyclic group, we get
	\begin{equation}\label{no-of-generator-equation}
		|H|-|\Phi(G)|=\varphi(|H|), \mbox{ i.e., }(p-1)|\Phi(G)|=\varphi(|H|),
	\end{equation}
	where $\varphi$ denote the Euler's totient function.
	
	{\bf Claim 3:} Intersection of any maximal subgroup of $G$ with $H$ is $\Phi(G)$.\\
	{\it Proof of Claim 3:} Let $K$ be a maximal subgroup of $G$ other than $H$. Then $H \cap K\subsetneq H$. As proper subgroups of $H$ are contained in $\Phi(G)$, we have $H \cap K\subseteq \Phi(G)$. On the other hand, $\Phi(G)$ being the intersection of all maximal subgroups is contained in $H \cap K$. Thus $H\cap K=\Phi(G)$.  
	
	Thus, the intersection number of $G$, $\iota(G)=2$. As $G$ is nilpotent, by Theorem 4.7 of \cite{intersection-number}, we have $$2=\iota(G)=\sum_{i=1}^{k}\iota(P_i)=\sum_{i=1}^{k}r_i,$$ where $\pi(G)=\{p_1,p_2,\ldots,p_k\}$ and $P_i$ is a Sylow $p_i$-subgroup of $G$ with rank $r_i$. As the only partitions of $2$ are $1+1$ and $2+0$, $|\pi(G)|\leq 2$.
	
	{\bf Case 1: $2=1+1$.} In this case, $\pi(G)=\{p,q\}$ and $|G|=p^aq^b$ and $G\cong P\times Q$, where $P,Q$ are Sylow $p$ and $q$ subgroups of $G$ respectively of rank $1$ each, i.e., $|P/\Phi(P)|=p$ and $|Q/\Phi(Q)|=q$. Thus $$|G/\Phi(G)|=\dfrac{|G|}{|\Phi(G)|}=\dfrac{|P\times Q|}{|\Phi(P)\times \Phi(Q)|}=|P/\Phi(P)||Q/\Phi(Q)|=pq,$$ i.e., $|\Phi(G)|=p^{a-1}q^{b-1}$ and $|H|=p^aq^{b-1}$. Thus, if $b>1$, from Equation \ref{no-of-generator-equation}, we get $$(p-1)p^{a-1}q^{b-1}=\varphi(p^aq^{b-1})\Rightarrow q=q-1, \mbox{ a contradiction.}$$
	Thus, the only possibility is $b=1$, i.e., $|G|=p^aq, |H|=p^a, |\Phi(G)|=p^{a-1}$. Moreover, as $\Phi(G)$ is non-trivial, we have $a\geq 2$. Also, as $H$ is a maximal subgroup of a nilpotent group $G$, $H$ is normal in $G$. Hence, $H=P$ is the unique Sylow $p$-subgroup of $G$ and $G\cong H\times Q$, where $Q$ is the Sylow $q$-subgroup of $G$ of order $q$. Thus $G$ is a cyclic group of order $p^aq$, i.e., $G\cong \mathbb{Z}_{p^aq}$, with $a\geq 2$.
	
	{\bf Case 2: $2=2+0$.} In this case, $G$ is a $p$-group, i.e., say $|G|=p^n$ and $|H|=p^{n-1}, \Phi(G)=p^{n-2}$. Moreover $G$ can not be cyclic, as in that case $\Gamma^*(G)$ will be null graph. Thus, $G$ is a non-cyclic $p$-group with a cyclic subgroup of index $p$. Then by Theorem 1.2  \cite{prime-power-groups-book}, $G$ is isomorphic to one of the following groups:
	\begin{itemize}
		\item $\mathbb{Z}_{p^{n-1}}\times \mathbb{Z}_{p}$
		\item $M_{p^n}=\langle a,b: a^{p^{n-1}}=b^p=e; b^{-1}ab=a^{1+p^{n-2}} \rangle$
		\item $D_{2^{n}}=\langle a,b: a^{2^{n-1}}=b^2=e; bab=a^{-1} \rangle$
		\item $Q_{2^{n}}=\langle a,b: a^{2^{n-1}}=e; b^2=a^{2^{n-2}}; b^{-1}ab=a^{-1} \rangle$
		\item $SD_{2^{n}}=\langle a,b: a^{2^{n-1}}=b^2=e; bab=a^{-1+2^{n-2}}\rangle$.
	\end{itemize}
	The converse part follows immediately from the following observations: 
	\begin{itemize}
		\item $\mathbb{Z}_{p^{n-1}}\times \{[0]\}$ is an universal vertex in $\Gamma^*(G)$, when $G \cong \mathbb{Z}_{p^{n-1}}\times \mathbb{Z}_{p}$.
		\item $\langle a \rangle$ is an universal vertex in $\Gamma^*(G)$, when $G \cong M_{p^n},D_{2^{n-1}},Q_{2^{n}}$ or $SD_{2^{n}}$.
	\end{itemize} 
\end{proof}

\begin{corollary}
	Let $G$ be a finite nilpotent group. The domination number of $\Gamma^*(G)$ is $1$ if and only if $G$ is one of the groups mentioned in Theorem \ref{Gamma*-universal}.\qed 
\end{corollary}

\begin{theorem}\label{bipartite-theorem}
	Let $G$ be a finite nilpotent group. Then $\Gamma(G)$ is bipartite if and only if $G$ is a cyclic group of order $p^a$ or $p^aq^b$, where $p,q$ are distinct primes.
\end{theorem}
\begin{proof}
	By Theorem \ref{connected-component}, it follows that $\Gamma(G)$ is connected, except a few possible isolated vertices. As isolated vertices does not affect the bipartiteness of a graph, we ignore the isolated vertices. If $G$ has a unique maximal subgroup, then $G$ is cyclic $p$-group and $\Gamma(G)$ is edgeless and hence bipartite.
	
	If $G$ has at least $3$ maximal subgroups, say $M_1,M_2$ and $M_3$, then, as $G$ is nilpotent, we have $M_1M_2=M_2M_3=M_3M_1=G$, i.e., we get a $3$-cycle $M_1\sim M_2\sim M_3\sim M_1$. Hence $\Gamma(G)$ is non-bipartite.	
	
	So we assume that $G$ has exactly two maximal subgroups $M_1$ and $M_2$. Then, by Proposition \ref{preli-result}(2), $G$ is a cyclic group of $p^aq^b$, where $p,q$ are distinct primes, i.e., $G \cong \mathbb{Z}_{p^aq^b}$. Then the vertices of $\Gamma^*(G)$ can be partitioned into two partite sets $V_1=\{\langle p \rangle, \langle p^2 \rangle,\ldots, \langle p^{a} \rangle \}$ and $V_2=\{\langle q \rangle, \langle q^2 \rangle,\ldots, \langle q^{b} \rangle \}$, thereby making it bipartite. 
\end{proof}	

\begin{theorem}\label{girth-theorem}
	Let $G$ be a finite nilpotent group. Then $$\mbox{girth}(\Gamma^*(G))=\left\lbrace \begin{array}{lll}
		\infty, & \mbox{ if }G \cong \mathbb{Z}_{p^a} \mbox{ or }\mathbb{Z}_{p^aq}, \mbox{ where }a\geq 1& \\
		4, & \mbox{ if }G \cong \mathbb{Z}_{p^aq^b},\mbox{ where }a,b\geq 2&p,q \mbox{ are distinct primes.}\\
		3, & \mbox{ otherwise.}&
	\end{array}  \right.$$.
\end{theorem}
\begin{proof}
	By Corollary \ref{Gamma*-connected}, it follows that $\Gamma^*(G)$ is connected. If $G$ has a unique maximal subgroup, then $G$ is cyclic $p$-group and $\Gamma(G)$ is edgeless.
	
	If $G$ has at least three maximal subgroups, say $M_1,M_2$ and $M_3$, then, as $G$ is nilpotent, we have $M_1M_2=M_2M_3=M_3M_1=G$, i.e., we get a $3$-cycle $M_1\sim M_2\sim M_3\sim M_1$. Hence girth of $\Gamma^*(G)$ is $3$.	
	
	So we assume that $G$ has exactly two maximal subgroups $M_1$ and $M_2$. Then by Proposition \ref{preli-result}(2), $G \cong \mathbb{Z}_{p^aq^b}$ and by Theorem \ref{bipartite-theorem}, $\Gamma^*(G)$ is bipartite. If any of $a$ or $b$ is $1$, then $\Gamma(G)$ is a star and has no cycle, i.e., girth of $\Gamma(G)$ is $\infty$. If both $a,b\geq 2$, then we have the $4$-cycle $\langle p \rangle \sim \langle q \rangle\sim \langle p^2 \rangle\sim \langle q^2 \rangle\sim \langle p \rangle$ in $\Gamma^*(G)$ and hence girth of $\Gamma^*(G)$ is $4$.  
\end{proof}

\section{Conclusion and Open Issues}
In this paper, we continued the study of co-maximal subgroup graph $\Gamma(G)$ and introduced the deleted co-maximal subgroup graph $\Gamma^*(G)$ of a group $G$. We discuss its various properties like connectdeness, girth and bipartiteness. However, there are natural questions which are yet to be resolved.

The first question arises from Remark \ref{diam=4-remark}.

{\bf Question 1:} Does there exist a finite group $G$ such that $diam(\Gamma^*(G))=4$?

On the light of Corollary \ref{nilpotent-diameter-bound} and \cite{akbari} Theorem 2.2, we can say that if such a group $G$ exists, then it must be non-nilpotent and $\Gamma(G)$ must have isolated vertices.

The next question arises from Remark \ref{non-solvable-connected-remark}.

{\bf Question 2:} Does there exist a finite non-solvable group such that $\Gamma^*(G)$ is disconnected?

It is noted that $\Gamma(\mathbb{Z}_6)\cong \Gamma(\mathbb{Z}_{15})\cong K_2$, i.e., non-isomorphic groups can have isomorphic co-maximal subgroup graphs. This leads us to the next question.

{\bf Question 3:} Under what condition, $\Gamma(G_1)\cong \Gamma(G_2)$ implies $G_1\cong G_2$?

The results in this paper and in \cite{akbari}, mainly concludes about some graph properties of $\Gamma(G)$ and $\Gamma^*(G)$ from the group properties of $G$. However, in order to address the above question, we need to find some results in which we can get some information about the group $G$ from the graph properties of $\Gamma(G)$ and $\Gamma^*(G)$. Some of such results will appear in a sequel of this paper.

\section*{Acknowledgement}
The first and second authors acknowledge the funding of DST-SERB-SRG Sanction no. $SRG/2019/000475$ and $SRG/2019/000684$, Govt. of India. The first author also acknowledges the financial support from DST-FIST [File No. $SR/FST/MS-I/2019/41$].

\section*{Conflict of Interests}
On behalf of all authors, the corresponding author states that there is no conflict of interest.

\section*{Statements \& Declarations}
\subsection*{Funding}
The first and second authors acknowledge the funding of DST-SERB-SRG Sanction no. $SRG/2019/000475$ and $SRG/2019/000684$, Govt. of India.
\subsection*{Competing Interests}
The authors have no relevant financial or non-financial interests to disclose.
\subsection*{Author Contributions}
All authors contributed to the study conception and design. Material preparation and analysis were performed by Manideepa Saha and Angsuman Das. The first draft of the manuscript was written by Manideepa Saha and all authors commented on previous versions of the manuscript. All authors read and approved the final manuscript.
\subsection*{Data availability}
Data sharing not applicable to this article as no datasets were generated or analysed during the current study.

\section*{References}
\begin{thebibliography}{9999}
	
	\bibitem{enhanced-power-graph} Aalipour G, Akbari S, Cameron P J, Nikandish R and  Shaveisi, F, On the structure of the power graph and the enhanced power graph of a group, Electron. J. Combin. 24(3) (2017), P3.16.
	\bibitem{non-commuting-graph} Abdollahi A, Akbari S and Maimani H R, Non-Commuting Graph of a Group, J. Algebra, Vol. 298, pp. 468-492, 2006. 
	\bibitem{akbari} Akbari S, Miraftab B and Nikandish R, Co-maximal Graphs of Subgroups of Groups, Canad. Math. Bull., Vol. 60(1), pp.12-25, 2017.
	\bibitem{hiranya} Bera S, Dey H K and Mukherjee S K, On the Connectivity of Enhanced Power Graphs of Finite Groups, Graphs Combin. 37, pp. 591-603, 2021.
	\bibitem{prime-power-groups-book} Berkovich Y, Groups of Prime Power Order, Volume 1, De Gruyter Expositions in Mathematics, 46, 2008. 
	\bibitem{Cameron-survey} Cameron P J, Graphs defined on groups, Int. J. Group Theory, Volume 11, Issue 2, pp. 53-107, 2022.
	\bibitem{Cameron-Ghosh} Cameron P J and Ghosh S, The Power Graph of a Finite Group, Discrete Math., Vol. 311, pp. 1220-1222, 2011. 
	\bibitem{subgroup-inclusion-graph} Devi P and Rajkumar R, Inclusion graph of subgroups of a group, \url{https://arxiv.org/pdf/1604.08259.pdf}
	\bibitem{khazal} Khazal R R, A Note on Maximal Subgroups in Finite Groups, Kyungpook Math. J., Volume 31, No. 1, pp. 83-87, June, 1991.
	\bibitem{miller-moreno} Miller G A and Moreno H C, Non-abelian Groups in which every subgroup is abelian, Trans. Amer. Math. Soc., Vol. 4, pp. 398-404, 1903.
	\bibitem{commuting-graph} Morgan G and Parker C, The diameter of the commuting graph of a finite group with trivial centre, J. Algebra, 393, pp. 41-59, 2013.
	\bibitem{rotman-book} Rotman J J, An Introduction to the Theory of Finite Groups, 4th Edition, Graduate Text in Mathematics, Springer, 1995.
	\bibitem{manideepa-sucharita-angsu} Saha M, Biswas S and Das A, On co-maximal subgroup graph of $\mathbb{Z}_n$, Int. J. Group Theory, Article in Press, DOI: \url{10.22108/IJGT.2021.129788.1732}
	\bibitem{intersection-number} Serrano H B, On the Intersection Number of Finite Groups, M.S. Thesis, University of Texas, 2019. Available at:  \url{https://scholarworks.uttyler.edu/cgi/viewcontent.cgi?article=1009&context=math_grad}
	\bibitem{minimal-non-cyclic-paper} Shi J and Zhang C, A Note on Finite Groups in which all non-cyclic proper subgroups have the same order, Indian J. Pure Appl. Math., 47(4), pp.687-690, December 2016. 
	\bibitem{west-graph-book} West D.B., Introduction to Graph Theory, Prentice Hall, 2001.
	
\end{thebibliography}
\end{document}